\documentclass[a4paper,11pt,oneside]{amsart}

\usepackage{geometry}
\usepackage{amssymb}  
\usepackage{latexsym} 
\usepackage{comment}
\usepackage{color}
\usepackage{colonequals}
\usepackage{amsmath}
\usepackage{booktabs}
 \usepackage{epsfig}

\usepackage{tikz}
\usepackage{tikz-cd}

\usepackage{textcmds} 

\usepackage{multirow}

\usepackage{setspace}
\definecolor{darkgreen}{rgb}{0,0.5,0}
\usepackage[
        colorlinks, citecolor=darkgreen,
        pdfauthor={Ingrid Bauer, Christian Gleissner, Julia Kotonski},
        pdftitle={Rigid Manifolds},
        linktocpage        
]{hyperref}

\usepackage[alphabetic,backrefs,lite]{amsrefs} 

\usepackage[all]{xy} 

\usepackage{enumerate} 

\newtheorem{theorem}{Theorem}[section]
\newtheorem{proposition}[theorem]{Proposition}

\newtheorem{corollary}[theorem]{Corollary}
\newtheorem{lemma}[theorem]{Lemma}

\newtheorem*{theorem-non}{Theorem}
\newtheorem*{question-non}{Question}
\newtheorem{openprob}[theorem]{Open Problem}

\theoremstyle{remark}

\newtheorem{rem}[theorem]{Remark}

\newenvironment{dedication}
  {
   \thispagestyle{empty}
   \itshape             
   \raggedleft          
  }
  {
  }

\theoremstyle{definition}
\newtheorem{definition}[theorem]{Definition}
\newtheorem*{general}{General Assumption}

\newcommand\Oh{{\mathcal O}}

\newcommand\ze{\zeta}

\newcommand{\CC}{\ensuremath{\mathbb{C}}}
\newcommand{\RR}{\ensuremath{\mathbb{R}}}
\newcommand{\ZZ}{\ensuremath{\mathbb{Z}}}

\newcommand{\PP}{\ensuremath{\mathbb{P}}}

\newcommand{\ra}{\ensuremath{\rightarrow}}

\newcommand{\Sym}{\operatorname{Sym}}

\DeclareMathOperator{\divi}{div}

\DeclareMathOperator{\Deck}{Deck}

\DeclareMathOperator{\PSL}{PSL}
\DeclareMathOperator{\PGL}{PGL}

\DeclareMathOperator{\Aut}{Aut}
\DeclareMathOperator{\Stab}{Stab}

\DeclareMathOperator{\diag}{diag}

\DeclareMathOperator{\GL}{GL}

\DeclareMathOperator{\cone}{cone}

\DeclareMathOperator{\ord}{ord}

\numberwithin{equation}{section}

\newcounter{nootje}
\renewcommand\check[1]
  {\marginpar{\tiny\begin{minipage}{20mm}\begin{flushleft}\thenootje : #1\end{flushleft}\end{minipage}}\addtocounter{nootje}{1}}
\begin{document}
\title[On Rigid Manifolds of Kodaira Dimension 1]{On Rigid Manifolds of Kodaira Dimension  1}

\author{Ingrid Bauer, Christian Gleissner, Julia Kotonski}

\thanks{
The second author wants to thank Stephen Coughlan and Andreas Demleitner for interesting and useful conversations about rigid manifolds.}

\begin{abstract} 
We discuss rigid compact complex manifolds of Kodaira dimension 1, arising as product-quotient varieties. First, we show that there is no free rigid action on the product of $(n-1)$ elliptic curves and a curve of genus at least two. Then, we describe the occurring groups, study the quotients and prove that there is always a suitable resolution of singularities preserving rigidity. Finally, we give a complete classification of the cases arising for minimal group orders.
\end{abstract}

\maketitle
\begin{dedication}
Dedicated to the memory of Alberto Collino. 
\end{dedication}
\tableofcontents

\section{Introduction}

A compact complex manifold with no non-trivial deformations is called \emph{rigid}. In the seminal paper \cite{rigidity},  Bauer and Catanese discussed and posed various open questions and problems regarding rigid manifolds with certain geometric properties, among others  the relation between rigidity and Kodaira dimension. 

Clearly, the only rigid curve is  $\PP^1$. 
In the above paper, it was shown that 
rigid surfaces occur only in Kodaira dimension $-\infty$ and $2$, see 
\cite{rigidity}*{Theorem 1.3}.
Moreover, the authors showed that examples of rigid manifolds in any dimension $n\geq 3$ and any Kodaira dimension $\kappa= -\infty, 0, 2, 3,\ldots , n$, can easily be constructed. This indicates that in higher dimensions, rigid varieties should be more frequent and should appear with any  Kodaira dimension (see also \cite{beau} for $\kappa=0$). 

The missing case $\kappa=1$ was left as an open problem:

\begin{question-non}
Do there exist rigid compact complex manifolds of dimension $n \geq 3$ and Kodaira dimension $1$?
\end{question-non}

In  \cite{BG20}, the first two authors gave a positive answer to this question by the following  ad-hoc construction: 
let $F$ be the Fermat cubic curve and $Q$ the Klein quartic. The unique non-abelian group $G$ of order $21$ acts on these curves in a way such that the diagonal action on the product $F^{n-1} \times Q$ leads to a rigid singular product-quotient variety 
\[
X_n = (F^{n-1} \times Q) /G \qquad \makebox{for} \qquad n\geq 3.
\]
Moreover, the authors showed that there is  a  resolution of singularities $\rho \colon \hat{X}_n \to X_n$
such that $\hat{X}_n$ is still rigid and has Kodaira dimension $1$.

It seems  natural  to ask whether it is possible to find \qq{easier} examples, e.g., product-quotients by \emph{free} actions, so that we can avoid to deal with singular varieties.

In the present article, we show that, in contrast to all other Kodaira dimensions,  this is not possible.

Furthermore, it would be desirable to have a classification of all rigid projective manifolds of Kodaira dimension one arising as a suitable resolution of singularities of product-quotient varieties as it was done in \cite{bauergleissner2} for the case of Kodaira dimension $0$.

Guided by the above example, we are looking for  $X$, the quotient of a product of $(n-1)$ elliptic curves and a curve of genus at least two by a diagonal action of a finite group $G$, a normal projective variety with isolated canonical quotient singularities, Kodaira dimension $1$ and 
$H^1(X, \Theta_{X}) = 0$.

By a result of Schlessinger and by Kuranishi theory,   it follows  that $X$ is a rigid (singular) variety.

Since we are looking for rigid {\em manifolds}, we construct a suitable resolution $\rho \colon \hat{X} \ra X$ of singularities and show that 
$H^1(X, \Theta_{X}) = H^1(\hat{X}, \Theta_{\hat{X}})$. 

Our first main result is:
\begin{theorem-non}
    Let $G$ be a finite group admitting a rigid diagonal action on a product 
    $$E_1\times \ldots\times E_{n-1}\times C,$$
    where the $E_j$ are elliptic curves, $C$ is a curve of genus at least two and the action is faithful on each factor. Then:
    \begin{enumerate}
        \item The elliptic curves $E_j$ are all the same and either isomorphic to $\CC/\ZZ[i]$ or $\CC/\ZZ[\ze_3]$.
        \item The group $G$ is a semi-direct product $A\rtimes \ZZ_d$, where $A$ is abelian and $\ZZ_d$ is the cyclic group of order $d$, where $d=3,4$ or $6$.
        \item The action of $G$ is never free. The isolated fixed points descend to cyclic quotient singula\-ri\-ties of type
        \[
            \frac{1}{\ell}(1,\ldots, 1) \qquad \rm{or} \qquad \frac{1}{\ell}(1,\ldots, 1,\ell-1),
        \]
        where $\ell$ divides $d$. They are canonical if $n\geq d$.
        \item For each quotient $X=E^{n-1}\times C$, there exists a resolution of singularities, $\rho\colon\hat{X}\to X$, preserving the rigidity. If $n\geq d$, then $\hat{X}$ is a rigid manifold of Kodaira dimension one. 
    \end{enumerate}
\end{theorem-non}

By the rigidity of the action,  the curves  $C$ and  $E$ in the theorem are realized as triangle curves, i.e., Galois $G$-covers of the projective line $\mathbb P^1$ branched in three points $p_1,p_2$ and $p_3$. To such a cover, we can attach a generating triple for the group $G$ via the monodromy map associated to the cover: the elements are the images $g_i$ of simple loops $\gamma_i$ around $p_i$. They generate the group $G=A\rtimes\ZZ_d$ and fulfill the relation $g_1\cdot g_2\cdot g_3=1$. They will play a crucial role to derive our classification results since a lot of geometric properties of triangle curves are encoded in the generating triples. For a more detailed discussion we refer to Section~\ref{se3} or to the  paper \cite{IFG}.

It turns out that the geometry of the rigid quotient varieties $(E^{n-1}\times C)/G$
differs slightly depending whether the generating triple $S_C$ corresponding to the curve $C$ has elements belonging to $A$ or not. We write for short $S_C\subset G\setminus A$ or $S_C\not\subset G\setminus A$, respectively.

The first case is special because it can only occur for $d=6$. Here, all rigid quotients are obtained as finite covers of a minimal one. More precisely, we have:

\begin{theorem-non}
    If $S_C\subset G\setminus A$, then there exists a finite holomorphic cover onto a unique minimal rigid quotient
    \[
        f\colon (E^{n-1}\times C)/G\longrightarrow X_{min}:=(E^{n-1}\times C')/\ZZ_6,
    \]
    where $C'$ is the hyperelliptic curve 
    \[
    C':=\{y^2=x_0^6+x_1^6\}\subset\PP(1,1,3)
    \]
    of genus two 
    and $E=\CC/\ZZ[\ze_3]$ is the Fermat elliptic curve.
\end{theorem-non}

In the general case $S_C\not\subset G\setminus A$, we can classify  the minimal examples, i.e., the rigid  quotients $(E^{n-1} \times C)/G$, where the group $G=A \rtimes_{\varphi_d} \mathbb Z_d$ 
has the smallest group order. Among them, there is also the example of \cite{BG20} that we sketched above.

\begin{theorem-non}
    The smallest groups of the form $A\rtimes\ZZ_d$ allowing a rigid diagonal action on $E^{n-1}\times C$ such that the $G$-action on $C$ has a generating triple $S_C\not\subset G\setminus A$ are:
    \begin{itemize}
        \item $G_3=\langle s,t  ~ \big\vert ~ s^3=t^7=1,~ sts^{-1} =t^4 \rangle$,
        \item $G_4=\langle s,t ~ \vert ~ s^4=t^5 =1, ~ sts^{-1}=t^3 \rangle$, 
        \item $G_6=\langle s,t  ~ \big\vert ~ s^6=t^3=1,~ sts^{-1} =t^2 \rangle$. 
    \end{itemize}
    If $d=6$, then there is precisely one isomorphism class of such quotients.\\
    If $d=3,4$, then there are at most two isomorphism classes of such quotients, which can be identified under complex conjugation.\\
    In each case, the curves $C=C_d$  are uniquely determined up to isomorphism: $C_3$ is the Klein quartic, $C_4$ is Bring's curve and $C_6$ is a smooth curve on the Fermat cubic surface with equation
    \[
    x_0x_2+x_1x_3=0.
    \]
\end{theorem-non}
Unfortunately, we are not able to decide whether the complex conjugate  varieties in the above theorem (for $d=3,4$) are biholomorphic or not. We pose this as an open problem.

The paper is organized as follows: in Section \ref{se2}, we briefly discuss  rigid group actions  on complex manifolds, in particular diagonal actions on products of curves.  Section \ref{se3} is about the basic theory of {\em triangle curves}, i.e., Galois covers of the projective line $\mathbb P^1$  branched in three points. 
Here, we also introduce spherical gene\-rating triples of a finite group $G$, which is a convenient notion, commonly used to describe triangle curves from a group theoretical point of view.  
In Section \ref{se4}, we provide a detailed analysis of finite groups $G$ affording  a rigid diagonal action on a product of $(n-1)$ elliptic curves and a curve $C$ of genus at least two. These groups are always semi-direct products $A\rtimes \ZZ_d$, where $A$ is abelian and $d=3,4$ or $6$. We show that one of the elements in the  spherical generating system $S_C$ attached to the triangle curve $C$ must be  contained  in $A$, except possibly for  $d=6$. Here,
the case 
$S_C\subset G\setminus A$ can also occur. 
This exceptional case is discussed in Section \ref{se5}, while Section \ref{se6} is devoted to the general case $S_C \not\subset  G\setminus A$. In these two sections, we complete the proofs of our main theorems, except for the existence of resolutions of the singularities that preserve  the rigidity.  Such resolutions are constructed 
in the final Section \ref{sec:resolutions} using  methods from toric geometry that generalize the ideas of our previous work in {\cite{BG20}.


\section{Rigid Group Actions}\label{se2}

\noindent
In this section, we recall some notions and results concerning {\it rigid group actions on complex mani\-folds}. 
\begin{definition}
Let $X$ be a compact complex manifold, let $\Theta_X$  be its tangent sheaf and let $G$ be a finite group acting holomorphically  on $X$. We say that the action is {\it rigid} if and only if 
$H^1(X,\Theta_X)^G=0$. 
\end{definition}

\begin{rem}\label{firstremark} \
\begin{enumerate}
    \item To be precise, in the previous definition, the correct wording should be that the action is {\it infinitesimally rigid} because $H^1(X,\Theta_X)^G$ parametrizes the first order infinitesimal $G$-invariant deformations of $X$. But since \qq{infinitesimal rigidity} is the only type of rigidity we are considering in this paper, we say \qq{rigid} by a slight abuse of notation.
    \item If $G$ acts freely in codimension one, then there are isomorphisms
    $$H^i(X/G,\Theta_{X/G}) \cong H^i(X,\Theta_X)^G,\qquad\makebox{for all}
    \qquad i \geq 0.$$
    In particular, if the action is rigid, the quotient $X/G$ has no infinitesimal {\it equisingular deformations}, i.e., no deformations preserving the singularities of $X/G$, since they are parameterized by $H^1(X/G,\Theta_{X/G})$. If $\dim(X)\geq 3$ and if the action has only isolated fixed points, then every infinitesimal deformation of the quotient preserves its singularities due to a result of Schlessinger \cite{schlessinger}. Thus, in this case, the quotient is in fact rigid, if the action is rigid. 
    \item Let $C$ be a compact Riemann surface and $G \leq \Aut(C)$ be a finite group. Then, the $G$-action on $C$ is rigid if and only if 
    $C/G \cong \mathbb P^1$ and the quotient map $C \to C/G \cong \mathbb P^1$ is branched in three points. Such curves are called  {\em triangle curves} and $f \colon C \rightarrow C/G \cong \mathbb P^1$ is called a {\em triangle cover}.
\end{enumerate}

\end{rem}

Let $G$ be a finite group acting holomorphically on the compact complex manifolds $X_1, \ldots, X_n$, then obviously, the cartesian product $G^n$ acts on the the product $X_1 \times \ldots \times X_n$, hence also the diagonal 
$\Delta_G \cong G \leq G^n$ by the formula
$$
g(x_1, \ldots, x_n):=(g x_1, \ldots, g x_n).
$$
We call this action  the {\it diagonal} $G$-action. K\"unneth's formula allows us to derive a criterion for the rigidity of the diagonal action (see \cite{bauergleissner2}, Proposition 2.3.): 

\begin{proposition}\label{rigiddiag}
Let $G$ be a finite group acting holomorphically on  the compact complex manifolds $X_1, \ldots, X_n$. Then,  the diagonal action on
 $X_1 \times \ldots \times X_n$ is rigid if and only if: 
\begin{enumerate}
\item
the $G$ action on each $X_i$ is rigid and 
\item 
$\big(H^0(X_i,\Theta_{X_i}) \otimes H^1(X_j,\mathcal O_{X_j})\big)^G=0$ for all $i \neq j$.
\end{enumerate}
\end{proposition}

In this paper, we are mainly interested in the special case where the complex manifolds $X_i$ are compact Riemann surfaces $C$. There, it is convenient to rephrase the rigidity conditions of Proposition~\ref{rigiddiag} in terms of the characters $\chi_C$ of the canonical representation
\[
\rho_C\colon G \to \GL(H^0(C,\omega_C)), \quad g \mapsto [\alpha \mapsto (g^{-1})^{\ast}(\alpha)].
\]

\begin{corollary}\label{charconds}
Let $G$ be a finite group acting holomorphically on the elliptic curves $E_1,  \ldots, E_n$ and on the curves 
$C_{1}, \ldots, C_m$ of genus at least two, then the diagonal action on the product 
$$E_1 \times \ldots \times E_n \times C_{1} \times \ldots \times C_m$$
is rigid if and only if 
\begin{enumerate}
\item 
the $G$-action on each $E_i$ and $C_j$ is rigid,
\item
$\langle  \chi_{E_i}  \cdot \chi_{C_j}, \chi_{triv}\rangle =0 $ for all $1 \leq i \leq n$ and $1 \leq j \leq m$, and 
\item
$\chi_{E_i} \cdot \chi_{E_j} \neq \chi_{triv}$ for all $1 \leq i < j \leq n$.
\end{enumerate}
\end{corollary}

Recall that $\chi_{triv}$ denotes the trivial character and $\langle -,-\rangle$ is the inner product in the space of class functions of $G$. 

\begin{proof}
By Serre duality and since  $H^0(\Theta_{C_j})=0$, it follows that condition $(2)$ of Proposition \ref{rigiddiag} is equivalent to 
$$
\big(H^1(\omega_{E_i}^{\otimes 2}) \otimes H^0(\omega_{C_j})\big)^G =0  \  \rm{and} \  
\big(H^1(\omega_{E_i}^{\otimes 2}) \otimes H^0(\omega_{E_j})\big)^G =0. 
$$
 By Dolbeault's interpretation of cohomology 
$$
H^1(\omega_{E_i}^{\otimes 2}) \simeq H_{\overline{\partial}}^{1,1}(E_i,\omega_{E_i}) = \langle (dz \wedge d\overline{z}) \otimes dz \rangle
$$
and the fact that the $G$-action on  $dz  \wedge d\overline{z}$ is trivial, it follows that the character of the representation $G \to \GL(H^1(\omega_{E_i}^{\otimes 2}))$
is equal to $\chi_{E_i}$.  This proves the claim.
\end{proof}


\section{Group Theoretical Description of Triangle Curves}\label{se3}

In what follows, we briefly recall the theory of triangle covers from the group theoretical point of view. As mentioned in Remark~\ref{firstremark}(3), triangle curves are finite Galois covers
of the projective line branched on three points $\mathcal B:=\lbrace p_1,p_2,p_3 \rbrace \subset \mathbb P^1$.
The fundamental group 
$\pi_1(\mathbb P^1 \setminus \mathcal B, \infty)$ is generated by three simple loops $ \gamma_1,  \gamma_2$ and  $\gamma_3$ around the  points $p_1, p_2$ and $p_3$, which fulfill a single relation, namely
$$
\pi_1(\mathbb P^1 \setminus \mathcal B, \infty ) = \langle \gamma_1, \gamma_2, \gamma_3 ~ | ~ \gamma_1 \cdot \gamma_2  \cdot \gamma_3  =1 \rangle. 
$$


 
\begin{definition}
Let $G$ be a finite group. A triple $S=[g_1,g_2,g_3]$ of non-trivial group elements is called a {\em spherical triple of generators} or shortly a  \emph{generating triple} of $G$ if
\[
    G= \langle g_1,g_2,g_3\rangle\qquad\makebox{and} \qquad g_1\cdot g_2 \cdot g_3=1_G.
\]
The \emph{type of $S$} is  defined as $t(S):=[\ord(g_1),\ord(g_2),\ord(g_3)]$. 
\end{definition}

Observe that a  generating triple $S=[g_1,g_2,g_3]$ of a finite group $G$ defines a surjective homomorphism
$\eta_S \colon \pi_1(\mathbb P^1 \setminus \mathcal B, \infty ) \to G$, $\gamma_i \mapsto g_i$. By \emph{Riemann's existence theorem}, the homomorphism $\eta_S$ induces a Galois triangle cover $f_S \colon (C_S,q_0) \to (\mathbb P^1,\infty)$ with branch locus $\mathcal B$ together with a unique isomorphism $\psi \colon G \to \Deck(f_S)$ such that the composition 
\[
(\psi \circ \eta_S) \colon \pi_1(\mathbb P^1 \setminus \mathcal B, \infty ) \to  \Deck(f_S)
\]
is the monodromy map of the associated unramified cover. For details, we refer to the textbook \cite{miranda}*{Section~4}.


\begin{rem}\label{triangle}\
\begin{enumerate}
    \item A point $q\in C_S$ has non-trivial stabilizer if and only if it belongs to one of the fibres $f_S^{-1}(p_i),\,i=1,2,3$. In this case, its stabilizer group has order $n_i=\ord(g_i)$.
The genus of $C= C_S$, the order of $G$ and the  $n_i$ are related by  \emph{Hurwitz's formula}: 
$$
2g(C)-2=|G|\left(1- \frac{1}{n_1} - \frac{1}{n_2} - \frac{1}{n_3} \right).
$$
For this reason, $t(S)$ is also called the \emph{branching signature} of  $f_S$.
    \item Applying a projectivity of the base $\mathbb P^1$, we can and will assume $p_1=-1, p_2=0$ and $p_3=1$. Moreover, we shall use generators $\gamma_i,\,i=1,2,3,$ of $\pi_1(\mathbb P^1 \setminus \mathcal B, \infty )$ such that the complex conjugate $\bar{\gamma}_i$ is the inverse of $\gamma_i$ (for example, those described in \cite{IFG}*{p.~7}).
\end{enumerate}
\end{rem}

It is important to understand when two generating triples give the same triangle cover. 

\begin{definition}
A \emph{twisted covering isomorphism} of two triangle $G$-covers 
$f_i\colon C_i \to \mathbb P^1$, $i=1,2$,
branched on 
$\mathcal B=\lbrace -1,0,1 \rbrace$, is a pair  $(u,v)$ of biholomorphic maps
$$
u \colon C_1 \to C_2, \ \  v \colon \mathbb P^1 \to \mathbb P^1 
$$
such that $v(\mathcal B) = \mathcal B$ and $u \circ  f_1 = f_2 \circ v$. 
\end{definition}

\begin{rem}
Let $\psi_i \colon G \to \Deck(f_i)$ be the corresponding $G$-actions, then the existence of a twisted covering isomorphism is equivalent to 
the existence of an automorphism $\alpha \in \Aut(G)$ and a biholomorphism $u\colon C_1 \to C_2$ such that 
\[
\psi_2(\alpha(g))\circ u = u \circ \psi_1(g) \quad \makebox{for all} \quad g \in G.
\]
As we shall see, this holds if and only if the corresponding generating triples belong to the same orbit of a certain group action on the set $\mathcal S(G)$ of all generating triples of $G$:\\
first of all, there  is a natural action of the  \emph{Artin-Braid} group  
$$
\mathcal B_3:=\langle \sigma_1,\sigma_2 ~ \vert ~ \sigma_1\sigma_2\sigma_1 =\sigma_2 \sigma_1 \sigma_2 \rangle 
$$
on $\mathcal S(G)$ defined by: 
\begin{itemize}
\item
$\sigma_1([g_1,g_2,g_3]):= [g_1g_2g_1^{-1},g_1,g_3]$, 
\item
$\sigma_2([g_1,g_2,g_3]):= [g_1,g_2g_3g_2^{-1},g_2]$.
\end{itemize}
This action commutes with the diagonal action of an automorphism  $\alpha \in \Aut(G)$ given by
$$\alpha ([g_1,g_2,g_3]):=[\alpha( g_1) , \alpha( g_2) ,\alpha( g_3) ].$$
Thus, we get  a well-defined action of the group $\Aut(G) \times \mathcal B_3$  on $\mathcal S(G)$. 
\end{rem}

The following result can be found in \cite{IFG}:

\begin{theorem}\label{CoveringIsomorphisms}
Let $G$ be a finite group and $S,S' \in \mathcal S(G)$ be two generating triples of $G$. Then, the following are equivalent:
\begin{enumerate}
\item
There is a twisted covering isomorphism of the covers $f_S\colon C_S \to \mathbb P^1$ and $f_{S'}\colon C_{S'} \to \mathbb P^1$.
\item
The generating triples $S$ and $S'$ are in the same $(\Aut(G) \times \mathcal B_3)$-orbit.
\end{enumerate} 
\end{theorem}

In the remaining part of the section, we shall apply the above to products of triangle curves.

A diagonal rigid action on a product of curves $C_1\times\ldots\times C_n$ such that the action on each curve is faithful corresponds to an $n$-tuple $[S_1,\ldots,S_n]$ of generating triples of $G$. In order to classify the quotients, we need a criterion when two given tuples of generating triples yield isomorphic quotients.\\
Let $[S_1,\ldots,S_n]$ and $[S'_1,\ldots, S'_n]$ be two tuples of generating triples of a finite group $G$. Under the assumption that all curves have genus at least two and that the diagonal action of $G$ on the product of the curves is free, the argument in \cite{cat00} generalizes to: every biholomorphism $f\colon X\to X'$ between the quotients lifts to a biholomorphism $\hat{f}\colon C_{S_1}\times\ldots\times C_{S_n}\to C_{S'_1}\times\ldots\times C_{S'_n}$ of the form
        \[
            (z_1,\ldots,z_n)\mapsto (u_1(z_{\tau(1)}),\ldots,u_n(z_{\tau(n)}))\qquad \makebox{with}\qquad \tau\in\mathfrak S_n.
        \]
If the action is not free or if there occur curves of genus at most one, then it is not clear that any isomorphism $f\colon X\to X'$ lifts, and if a lift exists, it is not necessarily of the above form. 
Nevertheless, the existence of an isomorphism having a lift as above can easily be checked using the generating triples and Theorem~\ref{CoveringIsomorphisms}, which gives us at least a sufficient criterion when two quotients are biholomorphic:

\begin{corollary}\label{bihollift}
    Let $[S_1,\ldots,S_n]$ and $[S'_1,\ldots, S'_n]$ be two tuples of generating triples of a finite group $G$. Then, the following are equivalent:
    \begin{enumerate}
        \item There exists a biholomorphism $f\colon X\to X'$ between the quotients which lifts to a biholomorphism $\hat{f}\colon C_{S_1}\times\ldots\times C_{S_n}\to C_{S'_1}\times\ldots\times C_{S'_n}$ of the form
        \[
            (z_1,\ldots,z_n)\mapsto (u_1(z_{\tau(1)}),\ldots,u_n(z_{\tau(n)}))\qquad \makebox{with}\qquad \tau\in\mathfrak S_n.
        \]
        \item There exist $\alpha\in\Aut(G)$, $\delta_1,\ldots,\delta_n\in \mathcal B_3$ and $\tau\in \mathfrak S_n$ such that 
        \[
            S'_j= \alpha(\delta_j(S_{\tau(j)})) \qquad \makebox{for all}\qquad j=1,\ldots,n.
        \]
    \end{enumerate}
\end{corollary}

Given a generating triple $S=[g_1,g_2,g_3]$ of $G$, the \emph{conjugate} of $S$ is defined as
\[
    \iota(S):=[g_1^{-1},g_1g_3,g_3^{-1}].
\]
In fact, since $\bar{\gamma}_i$ is the inverse of the path $\gamma_i$ (cf. Remark~\ref{triangle}), the conjugate triple yields the complex conjugate curve:

\begin{proposition}[\cite{IFG}, Proposition~2.3]
    For $S\in \mathcal S(G)$, it holds $\overline{C_S}=C_{\iota(S)}$.
\end{proposition}

\begin{rem}
    If $G$ is a finite group acting holomorphically on $X$, then, we obtain a natural holomorphic action of $G$ on the complex conjugate variety $\overline{X}$. 
    The  complex conjugate of the quotient $X/G$ is the same as the quotient of $\overline{X}$ by the natural $G$-action.
\end{rem}

\begin{corollary}
    The complex conjugate of the quotient corresponding to a tuple $[S_1,\ldots,S_n]$ of generating vectors of $G$ equals the quotient corresponding to $[\iota(S_1),\ldots,\iota(S_n)]$.
\end{corollary}


\section{Rigid actions on curves of genus $g\geq 1$}\label{se4}

In this section, we consider finite groups $G$ which admit a rigid action on curves of genus $g= 1$ and $g\geq 2$. For elliptic curves, i.e., in the case $g=1$, it is well-known that only very special groups can act faithfully on them. If we assume furthermore that this group $G$ also admits a faithful rigid action on a curve of genus $g\geq 2$, this will become even more restrictive.

Recall that the
automorphism group  of  an elliptic curve $E$ is a semidirect product 
$$\Aut(E)=E \rtimes \Aut_0(E),$$ 
where $\Aut_0(E) \cong \mathbb Z_2$, $\mathbb Z_4$ or $\mathbb Z_6$ (cf. \cite{miranda}*{Chapter III Proposition 1.12.}).
%


In \cite{bauergleissner2}*{Proposition~3.6}, the finite subgroups of $\Aut(E)$ allowing a rigid action on $E$ were classified:

\begin{proposition}\label{wallgrps}
A finite group $G$ admits a faithful rigid holomorphic action on an elliptic curve $E$ if and only if it is isomorphic to a semidirect product 
$$
A \rtimes_{\varphi_d} \mathbb Z_d, 
$$
where $d=3,4$ or $6$ and $A \leq \mathbb Z_n^2$ is a subgroup for some $n$,  invariant under the action 
$$
\varphi_d \colon  \mathbb Z_d \to \Aut\big(\mathbb Z_n^2\big), 
$$

defined  by:
\begin{itemize}
\item
$\varphi_3(1)(a,b)=(-b,a-b)$, 
\item
$\varphi_4(1)(a,b)=(-b,a)$ or
\item
$\varphi_6(1)(a,b)=(-b,a+b)$.
\end{itemize} The possible  branching signatures $[n_1,n_2,n_3]$ of the triangle cover $E \to E/G$, the  abelianizations of $G$ and the isomorphism types 
of $E$ are summarised in the table below:

{\begin{center}
\begin{tabular}{ c  c  c  c  }
  & $ $d=3$ ~ $ & $ ~ $d=4$ ~ $  & $ ~ $d=6$~ $   \\
 \hline \hline 
 $[n_1,n_2,n_3]$ & $ \quad   [3,3,3]  \quad   $ & $ \quad  [2,4,4]  \quad  $  & $ \quad  [2,3,6] \quad  $   \\

$G^{ab}$  & $ \quad \mathbb Z_3  $ or $ \mathbb Z_3^2   \quad $ & $ \quad \mathbb Z_4  $ or  $ \mathbb Z_2 \times \mathbb Z_4  \quad $  & $  \quad  \mathbb Z_6  \quad $  \\

  $E$ &  $\mathbb C/\mathbb Z[\zeta_3]$ & $\mathbb C/ \mathbb Z[i]$  &  $\mathbb C/ \mathbb Z[\zeta_3]$  \\
\hline

  \end{tabular}
  \end{center}
  }
\end{proposition}

An immediate geometric consequence of Proposition  \ref{wallgrps}  is:

\begin{corollary}\label{alliso}
Let $G$ be a finite group with a rigid  diagonal  action on a product  $ X \times E_1 \times \ldots \times E_n$ which is faithful on each factor, where $X$ is a compact complex manifold and $E_1, \ldots , E_n$ are elliptic curves. Then, the elliptic curves are all isomorphic to 
$\mathbb C/\mathbb Z[i]$ or they are all isomorphic to $\mathbb C/\mathbb Z[\zeta_3]$. Moreover, the branching signature $[n_1,n_2,n_3]$ is the same for each cover $E_i \to E_i/G$. 
\end{corollary}

We also need to recall the following:
\begin{proposition}[\cite{bauergleissner2}, Proposition 4.4.]\label{uniquetrans}
Let $G=A \rtimes_{\varphi_d} \mathbb Z_d$  be a finite group and $\psi \colon G \ra \Aut(E)$ be a rigid faithful action on an elliptic curve $E$.  Then, the translation subgroup of $G$, i.e., 
$$T_{\psi} :=\lbrace g \in G ~ \big\vert ~ \psi(g) ~\makebox{is a translation} \rbrace,$$
is always equal to $A$,  except 
if $G$ is one of the following: 
$$
\mathbb Z_3^2, \ \ \mathbb Z_3^2 \rtimes_{\varphi_3}  \mathbb Z_3, \  \  \mathbb Z_2 \times \mathbb Z_4 \ \ \rm{or} \ \  \mathbb Z_2^2 \rtimes_{\varphi_4}  \mathbb Z_4.
$$
\end{proposition}

These four groups will be called {\em exceptional} in the remaining part of the article. We will see later on that these groups don't admit rigid actions on a curve of genus $g\geq 2$.


Next, we present two  lemmata concerning 
the structure of semidirect products $A \rtimes_{\varphi_d} \mathbb Z_d$. The first one concerns the orders of elements not contained in $A$, the second one is a basic result about  subgroups of  semidirect products $\mathbb Z_n^2 \rtimes_{\varphi_d} \mathbb Z_d$. They turn out to be important for the classification of groups admitting a rigid diagonal action on a product $E^{n-1}\times C$, especially for finding such groups of minimal order.

\begin{lemma}[\cite{bauergleissner2}, Lemma 3.9.] \label{orderel}
The order of an element of $A \rtimes_{\varphi_d} \mathbb Z_d$ which is not contained in $A$ is equal to the order of its image under 
the canonical projection $ A \rtimes_{\varphi_d} \mathbb Z_d \to \mathbb Z_d$.  
 \end{lemma}

\begin{lemma}\label{cyclicinvar}
Let $p$ be a prime number.
\begin{enumerate}
\item
If $d=4$, then there exists a $\varphi_4$-invariant cyclic subgroup of $\mathbb Z_p^2$ of order $p$ if and only if $p=2$ or $p$ is odd 
and $4 ~ \big\vert ~  (p-1)$. 
\item
If $d=3$ or $6$, then there exists a $\varphi_d$-invariant cyclic  subgroup of  $\mathbb Z_p^2$ of order $p$  if and only if 
$p=3$ or $p>3$ and $3 ~ \big\vert ~  (p-1)$.
\end{enumerate}
\end{lemma} 

\begin{proof}
Assume  $d=4$. We show that an invariant subgroup exists if and only if the  equation 
$x^2 =-1$ has a solution in $\mathbb Z_p$. 
Suppose $c$ is a solution, then the subgroup
$\mathbb Z_p \simeq \langle (1,c) \rangle$ is $\varphi_4$-invariant since 
$\varphi_4(1)(1,c)=(-c,1)=-c(1,c)$. 
Conversely, let $\mathbb Z_p \simeq\langle (a,b) \rangle$ be a $\varphi_4$-invariant subgroup. Then, since $a \neq 0$, the element $a^{-1}\cdot (a,b) = (1,a^{-1}b)$ is also a  generator for this group and the $\varphi_4$-invariance implies that $a^{-1} b$ is a solution of the equation $x^2=-1$. 
Now for $p=2$, the equation $x^2 =-1$ has a solution. If $p$ is odd, there is a solution if and only if the Legendre symbol has value one, i.e.,
$1= \left({\tfrac {-1}{p}}\right)= (-1)^{\tfrac{p-1}{2}}$, where the second equality is given by Euler's criterion. This equality holds if and only if 
$4 ~ \big\vert ~  (p-1)$. 

For $d=3$ or $d=6$, the existence of an invariant subgroup is equivalent to the existence 
of a solution $c \in \mathbb Z_p$  for the equation $x^2-x+1=0$ or  
$x^2+x+1=0$, respectively. 
In each case, a solution exists if and only if $p=3$ or $p>3$ and 
$\left({\tfrac {-3}{p}}\right)=1$.
By  Euler and by quadratic reciprocity, we obtain that 
\[ \left({\tfrac {-3}{p}}\right)=
(-1)^{\tfrac{p-1}{2}}\cdot \left({\tfrac {3}{p}}\right) =\left(\tfrac {p}{3}\right)=1\iff 3~ \big\vert ~  (p-1).\qedhere
\]
\end{proof}

\begin{rem}\label{CyclicN}
If  $\mathbb Z_n^2$ has a cyclic  $\varphi_d$-invariant subgroup of order $n$, then for each prime divisor  $p$ of $n$, the group 
$\mathbb Z_p^2$ has a cyclic invariant subgroup of order $p$. In particular, $p$ has to fulfill the conditions of lemma \ref{cyclicinvar}.
\end{rem}

The next proposition analyzes which of the groups of Proposition \ref{wallgrps}, i.e., groups  admitting a faithful rigid action on an elliptic curve, also admit a faithful rigid action on a curve of genus at least two. In fact, we describe generating triples yielding triangle curves of genus at least 2, also correcting a flaw in 
\cite{IFG}*{Proposition~6.5}

Assume that $G =A \rtimes_{\varphi_d} \mathbb Z_d$ and  let 
$S = [g_1,g_2,g_3]$ be a generating triple for $G$. By abuse of notation, we shall write $S  \not\subset G\setminus A$  if $\{g_1,g_2,g_3\}  \not\subset G\setminus A$, and $S  \subset G\setminus A$ otherwise.

\begin{proposition}\label{genshapes}
Assume that $G =A \rtimes_{\varphi_d} \mathbb Z_d$, where $d \in \{3,4,6\}$, admits a faithful rigid action on a curve of genus $g(C) \geq 2$ and  let 
$S$ be a generating triple.
\begin{enumerate}
 \item If $d=6$ and $S \subset G\setminus A$, then $S$ is of type $[3,6,6]$ and 
of the form 
$[s^4h,sk,sc]$, where $s$ is  a generator of $\mathbb Z_6$ and $h,k,c \in A$.
\item
If $d=3$ or $4$, then $S \not\subset G\setminus A$.
\item 
If  $d$ is arbitrary and $S  \not\subset G\setminus A$, then 
$S$ is of type $[d,d,\ell]$ and $S$ is of the form $[sh,s^{-1}k,c]$, where $s$ is  a generator of $\mathbb Z_d$ and $h,k,c \in A$.
\end{enumerate}
\end{proposition}

\begin{proof}
We start with proving the third statement. By assumption, one of the entries of  $S=[g_1,g_2,g_3]$ is contained in $A$. W.l.o.g., we can assume $g_3$ to be this element. 
The relation  $g_1 \cdot g_2 \cdot g_3 =1$ implies that the 
equality $\overline{g_1}= \overline{g_2}^{-1}$ holds in the quotient  $G/A \simeq \mathbb Z_d$.
Thus,
$\overline{g_1}$ is a generator of $G/A \simeq \mathbb Z_d$, and the claim follows from Lemma \ref{orderel}.\\
In order to prove the first and second part of the proposition, we assume $S \subset G \setminus A$.  By Hurwitz's formula, we have the inequality
\begin{equation}\label{ineq}
1 - \frac{1}{\ord(g_1)} - \frac{1}{\ord(g_2)} - \frac{1}{\ord(g_3)}  > 0. 
\end{equation}
Since the orders of the elements $g_i$ coincide with the orders of their classes in $G/A \simeq \mathbb Z_d$, the triple
$[\overline{g_1},\overline{g_2},\overline{g_3}]$ is a generating triple for $G/A \simeq \mathbb Z_d$ of the same type as the type of $S$. 
In particular, the elements $\overline{g_i}$  are non-trivial and their orders divide $d$. This is impossible for $d=3$ by the above inequality \ref{ineq}. If $d=4$, then the only possible type fulfilling  inequality   \ref{ineq} is 
$[4,4,4]$. We can exclude this case since  $G/A \simeq \mathbb Z_4$ has no generating triple of this type.
If  $d=6$, the list of possible  types is:
    $[ 2, 6, 6 ]$, $[ 3, 3, 6 ]$, $[ 3, 6, 6 ]$   and 
    $[ 6, 6, 6 ]$. 
All but  $[ 3, 6, 6 ]$ can be excluded because  $\mathbb Z_6$ does not have a generating triple  of one of the other types. 
\end{proof}

As a direct consequence, we can exclude all the exceptional groups:

\begin{corollary}
    None of the exceptional groups (cf.  Proposition \ref{uniquetrans}) admits a rigid action on a curve of genus $g \geq 2$.
\end{corollary}

\begin{proof}
No exceptional group has a generating triple of type $[d,d,\ell]$ such that $1-\tfrac{2}{d}-\tfrac{1}{\ell} >0$. 
\end{proof}

With the further knowledge about the groups and its generating triples, we can investigate the rigidity of a diagonal action on a product $E^{n-1}\times C$ in more detail. It will turn out that second condition of Corollary~\ref{charconds} is almost always guaranteed.

\begin{general}
    From now on, when we talk about a \emph{diagonal action} of a group $G$ on a product $E^{n-1}\times C$, we implicitly assume that the action is faithful on each factor.
\end{general}

\begin{rem}
Assume that $G=A \rtimes_{\varphi_d} \mathbb Z_d$ admits a rigid diagonal action on $E^{n-1} \times C $, where $g(C)\geq 2$. Then, the canonical representation 
\[
\rho_E\colon G \to \GL\big(H^0(E,\omega_E)\big)
\]
must be the same for each copy of $E$.  Using Proposition~\ref{uniquetrans}, its character $\chi_{E}$  is  the composition of the quotient map $G \to G/A \simeq \mathbb Z_d$ and one of the characters
 $\chi_{\zeta_d}$ or $\chi_{\zeta_d^{-1}}$ . By abuse of notation, we identify  $\chi_E$ with the corresponding character  $\chi_{\zeta_d}$ or $\chi_{\zeta_d^{-1}}$, respectively. 
 \end{rem}

As a corollary, we can prove that in almost all cases, condition (2) of Corollary~\ref{charconds} is automatically fulfilled: 

\begin{corollary}\label{allbutonerigid}
Assume that $A \rtimes_{\varphi_d} \mathbb Z_d$ admits faithful rigid actions on an elliptic curve $E$ and on $C$, where $C$ has genus $\geq 2$. Then,
$\langle  \chi_{E}  \cdot \chi_{C}, \chi_{triv}\rangle =0 $, except for case $(1)$ in Proposition \ref{genshapes}.

In this case, the generating triple corresponding to the action on $C$ is of the form $S_C=[s^4h,sk,sc]$, and then, the action on the product is rigid if and only if $\chi_E(s) = \zeta_6^{-1}$.
\end{corollary}

\begin{proof}
Assume that  we are not in case (1) of Proposition \ref{genshapes}. Then, the generating triple $S_C$ giving the $G$-action on $C$ is of type $[d,d,\ell]$ and of the form $[sh,s^{-1}k,c]$, where $s$ is a generator of $\mathbb{Z}_d$ and $h,k,c \in A$. Observe that $\langle  \chi_{E}  \cdot \chi_{C}, \chi_{triv}\rangle =  \langle  \overline{\chi_{E}}  ,\chi_{C}\rangle$ and that moreover, by the formula of Chevalley-Weil (cf. \cite[Theorem~2.8]{FG16}), we get:
$$
\langle  \overline{\chi_{E}}  ,\chi_{C}\rangle = -1 + \tfrac{1}{d} + \tfrac{d-1}{d} =0.
$$
Assume now that  $d=6$ and the generating triple $S_C$ is contained in $G\setminus A$. Then, $S_C$ is of the form 
$[s^4h,sk,sc]$. If $\chi_E(s) = \zeta_6$, then 
$$
\langle  \overline{\chi_{E}}  ,\chi_{C}\rangle = \langle\chi_{\zeta_6^{-1}}, \chi_C\rangle =-1 + \tfrac{1}{3} + \tfrac{5}{6} + \tfrac{5}{6} =1, 
$$
but if  $\chi_E(s) = \zeta_6^{-1}$, then $\langle  \overline{\chi_{E}}  ,\chi_{C}\rangle  =0$.  
\end{proof}

Using  Proposition \ref{genshapes}, we can work out the possible types of singularities of our quotients and in particular see that a rigid diagonal action on $E^{n-1} \times C $  is never free.

\begin{corollary}\label{cor:sing}
Assume that 
$G=A \rtimes_{\varphi_d} \mathbb Z_d$  admits a rigid diagonal action on $E^{n-1} \times C $. Then, the quotient  
$X_n:=(E^{n-1} \times C)/G$ is singular and the  singularities are of type
$$
\frac{1}{\ell}(1,\ldots, 1) \quad \rm{or} \quad \frac{1}{\ell}(1,\ldots, 1,\ell-1), 
$$
where $\ell$ is a divisor of $d$.  In particular, $X_n$ has canonical singularities if $n \geq d$.
\end{corollary}

\begin{proof}
Since $S_C$ always contains at least one element that is not contained in $A$, we always find an element of $G$ having fixed points on $E$ and on $C$. Thus, the action is not free and $X_n$ must be singular.\\
Let $S_C =[sh,s^{-1}k,c] \not\subset G\setminus A$ and $p=(p_1, \ldots, p_n)\in E^{n-1} \times C$ be a point with non-trivial stabilizer and  $s^m a$ be a generator, where $a \in A$. 
Note that $m\neq 0$, as  $s^m a$ is not a translation.   By rigidity, the linear part of the action on $E^{n-1}$ is the same on each copy of $E$, and we may assume that $s$ acts on $E$ by multiplication  with  $\zeta_d $.
Since $s^m a$ is contained in $\Stab(p_n)$ and not in $A$, the stabilizer $\Stab(p_n)$ is a conjugate 
of $\langle sh\rangle$ or $\langle s^{-1}k\rangle$. In the first case, the  action of $s^ma$ around $p$ is 
$\diag(\zeta_d^m, \ldots, \zeta_d^m,\zeta_d^m)$ and in the second case $\diag(\zeta_d^m, \ldots, \zeta_d^m,\zeta_d^{-m})$. We conclude that $p$ descends to a singularity of type 
$$
\frac{1}{\ell}(1,\ldots, 1)  \quad \rm{or} \quad \frac{1}{\ell}(1,\ldots, 1,\ell-1), 
$$
where $\ell=\ord(s^ma)=\ord(\zeta_d^m)$ divides $d$. The criterion of Reid-Shepherd-Barron-Tai \cite{R87} tells us that these singularities are canonical if $n\geq d$.\\
The proof in the case $S_C\subset G\setminus A$ is similar.
\end{proof}

Since we are interested in rigid \emph{manifolds}, we have to provide resolutions of the rigid singular quotients $(E^{n-1}\times C)/G=X$ preserving the rigidity.

\begin{theorem}\label{theo:Resolutions}
    Any quotient $X_n:=(E^{n-1}\times C)/G$ by a rigid diagonal action of $G=A\rtimes_{\varphi_d}\ZZ_d$ has a resolution of singularities $\rho\colon\hat{X}_n\to X$ such that $H^1(\hat{X}_n,\Theta_{\hat{X}_n})=0$.\\
    If $n\geq d$,  $\hat{X}_n$ is a rigid manifold of Kodaira dimension one.
\end{theorem}

By Corollary~\ref{cor:sing}, the singularites of the quotients are isolated, hence, the construction of such resolutions is a local problem. Here, methods from toric geometry can be applied because germs of cyclic quotient singularities are represented by affine toric varieties.
We postpone the proof of the theorem to section~\ref{sec:resolutions}.


\section{The case $S_C \subset G \setminus A$}\label{se5}

The next two sections are devoted to analyze the rigid quotients of products $E^{n-1}\times C$ in more detail.

We start with the investigation of  case (1) of Proposition \ref{genshapes}, i.e., $d=6$ and the entries of
the generating triple $S_C$ for the $G$-action on $C$  are all contained in $G\setminus A$. 
It turns out that under this assumption,  all the possible  rigid quotients
$(E^{n-1} \times C)/G$ are  obtained as covers from a \emph{minimal one}, 
where $G=\mathbb Z_6$ and the curve $C$ has genus two.  

This minimal quotient is described as follows:
consider the hyperelliptic curve $$C':=\lbrace y^2=x_0^6+x_1^6\rbrace \subset \mathbb P(1,1,3)$$ of genus $2$  and consider Fermat's elliptic curve $E=\CC/\ZZ[\ze_3]$ together with the  rigid actions of $\ZZ_6$:
\begin{equation}\label{minimal}
s(x_0:x_1:y)= (x_0:\zeta_6 x_1: y) \qquad \rm{and} \qquad  s(z) = \zeta_6 z.
\end{equation}

Then, the induced diagonal action on $E^{n-1} \times C'$ is rigid and we denote the quotient by $X_{min}$.

More precisely, it holds:

\begin{theorem}\label{quot366}
Assume that $E^{n-1} \times C$ admits a rigid diagonal action of the group  $G=A \rtimes_{\varphi_6} \mathbb Z_6$ such that 
the $G$-action on $C$ has a generating triple $S_C \subset G\setminus A$.  Then:
\begin{enumerate}
\item 
The group $A$ acts freely on $C$ and the quotient $C':=C/A$ is isomorphic to the hyperelliptic curve 
\[
\lbrace y^2=x_0^6+x_1^6\rbrace \subset \mathbb P(1,1,3).
\]
The 
$G/A \simeq \mathbb Z_6$-action on $C'$ is given by 
$s(x_0:x_1:y)= (x_0:\zeta_6 x_1: y)$,
 up to the automorphism of $\mathbb Z_6$ exchanging  $s$ and $s^{-1}$. 
\item
The elliptic curve $E/A$ is isomorphic to $E$ and the induced 
$G/A \simeq \mathbb Z_6$-action on $ E^{n-1} \times C'$ 
is rigid and compatible with the $G$-action on $E^{n-1} \times C$. 
There is a finite holomorphic cover 
\[
f\colon (E^{n-1} \times C)/G \to (E^{n-1} \times C')/\mathbb Z_6 = X_{min}.
\]
of degree $\big\vert  A\big\vert^{n-1}$.
\item The singularities of $X_{min}=(E^{n-1} \times C')/\mathbb Z_6$ are:

 \begin{center}
{\scriptsize
\renewcommand{\arraystretch}{2.0}
\begin{tabular}{| c | c | c | c | c | }
\hline
type  & $\frac{1}{2}(1, \ldots, 1)$ &  $\frac{1}{3}(1, \ldots, 1)$ &  $\frac{1}{3}(1, \ldots, 1,2)$ & $\frac{1}{6}(1, \ldots, 1)$ \\
\hline number  & $\frac{2}{3}(4^{n-1}-1)$  & $3^{n-1}$  & $3^{n-1}-1$ & $2$ \\
\hline
\end{tabular}
}
\end{center}
\end{enumerate}
\end{theorem}

\begin{proof}
According to Proposition \ref{genshapes}, 
a generating triple  $S_C$ for the $G$-action on $C$  is of the form $[s^4h,sk,sc]$. Hence, the only group elements that can have fixed points on $C$ are contained in 
$G \setminus A$. This shows that  $A$ acts freely on $C$. Hurwitz's formula tells us that the genus of  $C'=C/A$ is two, and  thus, the curve is hyperelliptic.

Before we derive the equation of $C'$, we prove the second part of the theorem.

Clearly, $E/A$ is an elliptic curve, which is isomorphic to $E$. Moreover, 
the generating triple for the action of $G/A\simeq \mathbb Z_6$ on $C'$ is the image of $S_C$ under the quotient map $G \to \mathbb Z_6$ 
and therefore equal to $[s^4,s,s]$. Hence, by Corollary~\ref{allbutonerigid}, the action of $\mathbb Z_6$ on $E^{n-1} \times C'$ is also rigid since $g(C')=2$ and 
$\chi_E(s)=\zeta_6^{-1}$.\\
The degree of the induced cover $f$ equals $\lvert A\rvert^{n-1}$ as it fits in the following commutative diagram
\[
    \begin{tikzcd}
		E^{n-1}\times C \arrow{r}\arrow{d} & (E/A)^{n-1}\times (C/A)\simeq E^{n-1}\times C'\arrow{d}\\
	    (E^{n-1}\times C)/G \arrow{r}{f} & (E^{n-1}\times C')/(G/A)
	\end{tikzcd}
\]

Next, we determine the equation of $C'\subset \mathbb P(1,1,3)$, which must be of the form $y^2=f_6(x_0,x_1)$, where $f_6$ is homogenous of degree six. 
 Note that  $G/A\simeq \mathbb Z_6$ does not contain the hyperelliptic involution, 
otherwise the cover $\pi' \colon C'\to C'/\mathbb Z_6 \simeq \mathbb P^1$ would factor through the hyperelliptic cover
\[
\Xi \colon C' \to \mathbb P^1, \qquad (x_0:x_1:y) \mapsto (x_0:x_1). 
\]
This is impossible because  $\Xi$ has $6$ ramification points and $\pi'$ has only $3$.
Hence, $s$ descends to an automorphism $\widehat{s}$ of $\mathbb P^1$ of order six such that $\Xi \circ s = \widehat{s} \circ \Xi$. 
Up to a change of coordinates, $\widehat{s}$ is defined by
\[
\widehat{s}(x_0:x_1)= (x_0:\zeta_6^{\pm 1} x_1).
\]
This  forces $f_6$ to be  $f_6(x_0,x_1)=x_0^6+x_1^6$ up to multiplication of $x_0$ and $x_1$ by non-zero scalars.\\
The action of $s$ on $C'$ must be  one of the following:
\[
s(x_0:x_1:y)= (x_0:\zeta_6^{\pm 1} x_1: \pm y). 
\]
Up to permutation of $x_0$ and $x_1$, only the two possibilities
\[
    s(x_0:x_1:y)= (x_0:\zeta_6 x_1:y) \qquad \makebox{and} \qquad s(x_0:x_1:y)= (x_0:\zeta_6^{-1} x_1:y)
\]
remain. We claim that only the first choice for $s$ yields a rigid action on the product $E^{n-1}\times C'$.\\
By Corollary~\ref{charconds}, the action is rigid if and only if $\langle \chi_{C'},\overline{\chi}_E\rangle=0$. The character of the canonical representation of the curve $C'$ is easy to determine: 
on the open affine 
$x_0=1$, the curve is defined by $y^2=x^6+1$ and  a basis of  $H^0(C', \omega_{C'})$ is 
given by the 1-forms $\frac {dx}{y}$ and  $x \frac{dx}{y}$ for  $y \neq 0$.   
The pullback of this 1-forms with the inverse of $s$ yields  the character 
$\chi_{C'}=\chi_{\zeta_6^{-1}}+ \chi_{\zeta_6^4}$ for the first choice of the action and $\chi_{C'}=\chi_{\zeta_6}+ \chi_{\zeta_6^2}$ for the second. Since $\overline{\chi}_E(s)=\zeta_6$, the claim follows. 
 
Finally, we have to determine the singular points of the minimal quotient $X_{min}=(E^{n-1} \times C')/\mathbb Z_6$. For this, we need to know the points on $C'$ and on $E$ with non-trivial stabilizer and the 
action of the generator of the stabilizer-group in local coordinates.  The table below gives an overview: 

\bigskip

\begin{center}
{\footnotesize
{
\renewcommand{\arraystretch}{1.5}
\setlength{\tabcolsep}{4pt}
\begin{tabular}{lclc|}
\begin{tabular}{|c |}
\hline
point  $q$  \\
\hline
generator of $\Stab(q)$   \\
\hline
 local action   \\
\hline
\end{tabular}
&
\begin{tabular}{|c | c | c |}
\hline
$(1:0:\pm 1) $ & $(0:1:\pm 1)$   \\
\hline
$ s $ & $  s^2 $   \\
\hline
 $x \mapsto \zeta_6 x$ & $x \mapsto \zeta_6^4 x$  \\
\hline
\end{tabular}
&
\begin{tabular}{|c | c | c | }
\hline
$  0 $ & $ \pm\frac{2+\zeta_3}{3} $ & $ \frac{1}{2}, \frac{\zeta_3}{2}, \frac{1+\zeta_3}{2}$  \\
\hline
$  s   $ & $  s^2 $  & $ s^3 $ \\
\hline
 $x \mapsto \zeta_6 x$ & $x \mapsto \zeta_6^2 x$  & $x \mapsto -x$  \\
\hline
\end{tabular}
\end{tabular}
}}
\end{center}

\bigskip
\noindent 
The singularities of type  $\tfrac{1}{2}(1,\ldots, 1)$ are the images of the points 
\[
(z_1, \ldots,z_{n-1},p) \in E^{n-1} \times C'
\]
having a stabilizer group of order two.  Over each of these singularities, there are three points 
in the fibre. 
Using the table, we see that these points have coordinates 
\[
p \in \lbrace (1:0:\pm 1) \rbrace , \qquad  z_i  \in \left\lbrace 0 , \tfrac{1}{2}, \tfrac{\zeta_3}{2}, \tfrac{1+\zeta_3}{2}\right\rbrace,
\]
where at least one $z_i \neq 0$. The number of these points is  $2 (4^{n-1} -1)$, which implies that there are $\tfrac{2}{3} (4^{n-1} -1)$
singularities of type $\tfrac{1}{2}(1,\ldots, 1)$. 
The points on the product $E^{n-1} \times C'$ with stabilizer of order $6$ are 
$\left(0, \ldots,0,(1:0:\pm 1)\right)$. They descend to $2$ singular points of type $\frac{1}{6}(1, \ldots,1)$. 
For the points $(z_1, \ldots,,z_{n-1},p)$ with stabilizer of order $3$, we distinguish the cases
$|\Stab(p)|=3$ and $|\Stab(p)|=6$. In the first case, the coordinates of these points are 
\[
p \in \lbrace (0:1:\pm 1) \rbrace , \qquad  z_i  \in \left\lbrace 0,  \pm\tfrac{2+\zeta_3}{3} \right\rbrace. 
\]
These  $2 \cdot 3^{n-1}$ points descend to $3^{n-1}$ singularities of type $\tfrac{1}{3}(1, \ldots, 1,2)$. 
In the second case, where  $|\Stab(p)|=6$, we have  two choices for $p$  and $3$ choices for $z_i$, where 
at least one of the $z_i$ has stabilizer of order $3$, i.e., at least one  $z_i \neq 0$.  There are $2(3^{n-1}-1)$ of these points. They descend to 
$3^{n-1}-1$ singularities of type $\tfrac{1}{3}(1, \ldots, 1)$. 
\end{proof}


\section{The case $S_C \not\subset G \setminus A$}\label{se6} 

If  the action on $C$ has a generating triple $S_C \not\subset G \setminus A$, then 
Theorem~\ref{quot366} cannot hold without modification: 

\begin{rem}
Assume that $C$ admits a rigid action of the group  $G=A \rtimes_{\varphi_d} \mathbb Z_d$, with a generating triple  $S_C\not\subset G\setminus A$ of type $[d,d,\ell]$, 
 then $C/A$ is the projective line. 
 \end{rem}
 
 \begin{proof}
 Indeed, by Hurwitz's formula
\[
2g(C/A) -2 = d\left(-2 + 2 -\tfrac{2}{d} \right)= -2. \qedhere
\]
\end{proof}

However, we can mod out proper subgroups of $A$ to produce groups of smaller order. Under some assumptions, the quotient $C/A$ still has genus at least two, and we obtain a similar result as in Theorem~\ref{quot366}:

\begin{proposition}
Assume that $E^{n-1} \times C$ admits a rigid diagonal action of a group  $G=A \rtimes_{\varphi_d} \mathbb Z_d$, such that 
the $G$-action on $C$ has a generating triple of type $[d,d,\ell]$ of the form  $[sh,s^{-1}k,c]$. 
Let $A' \lneq A$ be a proper $\varphi_d$-invariant subgroup.  Then: 
\begin{enumerate}
\item
The group 
$G'=A/A' \rtimes_{\varphi_d} \mathbb Z_d$  has a generating triple of type $[d,d,\ell']$, where $\ell'$ is the order of $\overline{c}$ in $A/A'$.
\item
The quotient curve $C':=C/A'$ has genus at least two  if and only if 
\[
1- \tfrac{2}{d} - \tfrac{1}{\ell'} >0.
\]
\item
The  quotient $E/A'$ is isomorphic to $E$ and the induced $G'$-action on $E^{n-1} \times C'$ is rigid and compatible with the $G$-action on $E^{n-1} \times C$. 
There is a finite holomorphic cover 
\[
(E^{n-1} \times C)/G \to (E^{n-1} \times C')/G'.
\]
\end{enumerate}
\end{proposition}

\begin{proof} Let $\pi \colon G \to G'$ be the quotient map, then $$[\pi(sh),\pi(s^{-1}k),\pi(c)]= [s\overline{h}, s^{-1}\overline{k},\overline{c}]$$ is a generating triple of 
$G'$. Note that, since $A'\lneq A$, the class $\overline{c}$ is non-trivial. 
It corresponds to the  induced  $G'$-action on  $C'=C/A'$. Using Lemma~\ref{orderel}, we see that  the type of this generating triple  is  $[d,d,\ell']$, where 
$\ell'=\ord(\overline{c})$. \\
The second statement follows immediately from Hurwitz's formula, and for the last one, it suffices to note that  $E/A' \simeq E$. The rigidity of the induced action is clear by Corollary~\ref{allbutonerigid}. 
\end{proof}

In analogy to Theorem~\ref{quot366}, we want to find  for each $d=3,4, 6$ the minimal examples, i.e., the rigid  quotients $(E^{n-1} \times C)/G$ where the group $G=A \rtimes_{\varphi_d} \mathbb Z_d$ 
has the smallest group order.

\begin{lemma}\label{le:MinimalGroupsSnotinA}
The following three groups are the smallest groups of the form $A \rtimes_{\varphi_d} \mathbb Z_d$  which allow a 
faithful rigid and holomorphic action on a smooth curve $C$ of genus $g \geq 2$, such that the generating triple is  not contained in $G \setminus A$:  
\begin{itemize}
\item 
$G_3=\langle s,t  ~ \big\vert ~ s^3=t^7=1,~ sts^{-1} =t^4 \rangle$,
\item
$G_4=\langle s,t ~ \vert ~ s^4=t^5 =1, ~ sts^{-1}=t^3 \rangle$,
\item
$G_6=\langle s,t  ~ \big\vert ~ s^6=t^3=1,~ sts^{-1} =t^2 \rangle$. 
\end{itemize}
\end{lemma}

\begin{proof}
We only treat the case  $d=3$ because the arguments for  $d=4$ and $6$  are similar.  
Note that  $G_3$ admits the generating triple $[s,s^2t,t^6]$ of type $[3,3,7]$. Hence,  it is enough to 
exclude the groups $A \rtimes_{\varphi_3} \mathbb Z_3$ with  $|A| \leq 6$. 
By Hurwitz's formula: 
\[
0 < 1-\tfrac{2}{3} - \tfrac{1}{\ell},
\]
which implies $4 \leq \ell$, i.e., $A$ has an element of order at least four. Since $A$ has at most six elements, it must be 
cyclic of order $n=4$, $5$ or $6$. This is impossible by Remark \ref{CyclicN} because each $n$ has a  prime divisor 
$p \neq 3$  such that $3 \nmid (p-1)$.
\end{proof}

\begin{proposition}\label{minimalex}
For each of the groups $G_d$, there exists up to isomorphism a unique faithful rigid and holomorphic action on a smooth curve $C_d$. The curves $C_d$ are 
canonically embedded. Their equations in $\mathbb P^{g-1}$ and the actions by projective transformations  are given  in the table below: 
 \begin{center}
{\footnotesize
\begin{tabular}{| l |  l  | }
 \hline 
&  \\
d=3 & $G_3= 
 \langle s,t  ~ \big\vert ~ s^3=t^7=1,~ sts^{-1} =t^4 \rangle$  \\
& \\
& $C_3= \lbrace x_0^3x_1+x_1^3x_2+x_0x_2^3=0 \rbrace \subset \mathbb P^2$  \\ 
& \\
& $ s \mapsto 
\begin{pmatrix}
0 & 1 & 0 \\
0 & 0 & 1 \\
1 & 0 & 0 
\end{pmatrix},$   
 \qquad 
$ t \mapsto 
\begin{pmatrix}
\zeta_7^4 & 0 & 0 \\
0 & \zeta_7^2 & 0 \\
0 & 0 & \zeta_7 
\end{pmatrix}$  \\
& \\
\hline
&  \\

d=4 & $G_4 =\langle s,t ~ \vert ~ s^4=t^5 =1,~ sts^{-1}=t^3 \rangle$   \\
& \\
 & 
$C_4=\lbrace x_0x_3+x_1x_2 = 0, ~~ x_0^2x_2-x_0x_1^2+x_1x_3^2-x_2^2x_3 =  0\rbrace \subset \mathbb P^3$
\\
& \\
& 
$s \mapsto 
 \begin{pmatrix}
 0 & 0   & 1 & 0 \\
1 & 0  & 0  & 0  \\
0  & 0 & 0 & 1 \\
0 & 1 &  0 & 0 
  \end{pmatrix},$ \qquad  
$t \mapsto \begin{pmatrix}
\zeta_5 & 0 & 0 & 0 \\
0 & \zeta_5^2 & 0 & 0 \\
0 & 0 & \zeta_5^3 & 0 \\
0 & 0 & 0 &  \zeta_5^4
  \end{pmatrix}$ \\
& \\
\hline
& \\

d=6 & $G_6 =\langle s,t  ~ \big\vert ~ s^6=t^3=1,~ sts^{-1} =t^2 \rangle$   \\
&  \\
 & 
$C_6=\lbrace x_1x_3+x_0x_2=0, ~~ x_0^3+x_1^3+x_2^3+x_3^3=0\rbrace \subset \mathbb P^3$
\\
& \\
& 
$s \mapsto 
\begin{pmatrix}
0 & \zeta_3 & 0 & 0 \\
1 & 0 & 0 &  0 \\
0 & 0 & 0 & \zeta_3^2 \\
0 & 0 & 1 & 0 
\end{pmatrix},   \qquad 
t \mapsto 
\begin{pmatrix}
\zeta_3 & 0 & 0 & 0\\
0 & \zeta_3^2 & 0 & 0 \\
0 & 0 & \zeta_3^2 & 0 \\
0 & 0 & 0 & \zeta_3 
\end{pmatrix}.$
\\
& \\
\hline

  \end{tabular}
}
\end{center}
\end{proposition}

\bigskip

\begin{rem}
The curve $C_3$ is the \emph{Klein Quartic}, a curve of genus three with automorphism group $\PSL(2,\mathbb F_7)$, a finite simple group of order $168$. It is the largest group of automorphisms for a genus three curve (cf. \cite{Klein}). \\
The curve $C_4$ is isomorphic to  \emph{Bring's curve} in $\mathbb P^4$, a smooth curve of genus $4$  on the \emph{Clebsch diagonal cubic surface} defined by the equations:
\[
x_0^3+ \ldots + x_4^3= x_0^2+ \ldots + x_4^2=x_0+ \ldots + x_4=0. 
\]
A biholomorphic map from $C_4$ to Bring's curve is induced by 
\[
 \begin{pmatrix}
 -1 &  1  & 1 & -1 \\
-\zeta_5& \zeta_5^2  & \zeta_5^3  & -\zeta_5^4  \\
-\zeta_5^2 & \zeta_5^4  & \zeta_5 & -\zeta_5^3  \\
-\zeta_5^3 & \zeta_5 & \zeta_5^4  & -\zeta_5^2  \\
-\zeta_5^4 & \zeta_5^3  & \zeta_5^2  & -\zeta_5
 \end{pmatrix}
 \]
cf. \cite[proof of Theorem 9.5.8]{Dol12}. Note that the full automorphism group of Bring's curve is $\mathfrak S_5$, acting in the obvious way. It is the largest group of automorphisms for a genus $4$ curve (cf. \cite{wiman}).\\
A detailed description of the geometry of this curve can be found in the recent preprint \cite{disney}. 
\end{rem}

\begin{proof}[Proof of Proposition~\ref{minimalex}]
First, we show that the curves are canonical. 
 Suppose that one of the groups $G_d$ acts on  a hyperelliptic curve. Then, since the center $Z(G_d)$ 
contains no involution, the group $G_d$ embeds in  $\PGL(2,\mathbb C)$.
This is impossible because 
the finite subgroups of $\PGL(2,\mathbb C)$ are known to be cyclic, dihedral, $\mathfrak A_4$, $\mathfrak S_4$ and $\mathfrak A_5$.\\ 
Next, we shall derive equations for the curves and the actions. Here, we only treat $G_4$. The computations in case of the other two groups are similar. 
Since  $C_4 \subset \mathbb P^{g-1}$ is canonically embedded, the group $G_4$ acts by projective transformations. Note that $g(C_4)=4$ as the type of the generating triple of $G_4$ equals $[4,4,5]$.
The group $G_4$ has four representations of degree $1$, which are obtained from $G_4/\langle t \rangle \simeq \mathbb Z_4$ by inflation, and one irreducible representation of degree $4$ defined by:
\[
s \mapsto 
 \begin{pmatrix}
 0 & 0   & 1 & 0 \\
1 & 0  & 0  & 0  \\
0  & 0 & 0 & 1 \\
0 & 1 &  0 & 0 
  \end{pmatrix} \qquad  \makebox{and} \qquad 
t \mapsto \begin{pmatrix}
\zeta_5 & 0 & 0 & 0 \\
0 & \zeta_5^2 & 0 & 0 \\
0 & 0 & \zeta_5^3 & 0 \\
0 & 0 & 0 &  \zeta_5^4
  \end{pmatrix}.
\]
The canonical representation 
$G_4 \to GL(H^0(\omega_{C_4}))$ is therefore either the sum of four $1$-dimensional representations or equal to 
the irreducible representation 
of degree $4$. The first possibility can be ruled out since the canonical representation is faithful and $G_4$ is not abelian. 
This shows that the $G_4$-action on $C_4$ is given by the above  matrices because the vector spaces
$H^0(\omega_{C_4})$ and $\mathbb C[x_0, \ldots, x_3]_1$ are in natural bijection. 
The elements of order four in $G_4$ are conjugated either to  $s$ or to the inverse $s^{-1}$. 
As a projective transformation, $s$   has four fixed points:
\[
  p_1:=(1:i:-i:-1),\ p_2:=(1:-i:i:-1), \ p_3:=(1:1:1:1),\ p_4:=(1:-1:-1:1)
\]

Since $g=g(C_4)=4$, the curve $C_4$ is a complete intersection of a quadric $Q$ and a cubic $K$ (cf. \cite[p.~258]{GriffH}). To find $Q$ and $K$, we use Noether's classical theorem (cf. \cite[p.~253]{GriffH}), which says that the following sequence is exact for all $k\geq 1$:
\begin{equation}\label{noetherseq}
0 \to I(C_4)_k \to \mathbb C[x_0, \ldots, x_{3}]_k \to H^0(C,\omega_C^{\otimes k}) \to 0.  
\end{equation}
The vector space  $I(C_4)_2$ is $1$-dimensional and generated by $Q$, whereas $I(C_4)_3$ is $5$-dimensional and generated by $K$ together with the four reducible cubics $x_0Q, \ldots, x_3 Q$.  
Since the exact sequence \eqref{noetherseq} consists of compatible $G_4$-representations, they remain exact when  we  take the invariant parts. 
In particular, for $k=2$, we obtain the equality  
\[
I(C_4)_2^{G_4} = \mathbb C[x_0, \ldots, x_{3}]_2^{G_4} 
\]
by the rigidity of  the action. 
Since $\mathbb C[x_0, \ldots, x_{3}]_2$ is the symmetric square of the canonical representation, we easily find that 
$\mathbb C[x_0, \ldots, x_{3}]_2^{G_4}$ is one-dimensional, i.e., the quadric $Q$ is $G_4$-invariant. 
In particular, it  is invariant under the action of $t$ which is diagonal. This implies  that $Q$ is a linear combination of the monomials $x_0x_3$ and $x_1x_2$.   These monomials are  swapped  by $s$, which forces the quadric to  be 
\[
Q=x_0x_3 + x_1x_2,
\]
 up to a non-zero scalar. 
We immediately see that  $p_3$ and $p_4$ do not belong to  $Q$. Hence, the two $G$-orbits of points with stabilizer of order $4$ are represented by $p_1$ and $p_2$.\\
Next, we consider the sequence~\eqref{noetherseq} for $k=3$. As $Q$ is invariant  under $G_4$, 
the subspace of $I_3$ which is spanned by the products $x_0 Q, \ldots , x_3  Q$ is the irreducible $4$-dimensional representation of $G$.
The cubic $K$, on the other hand, yields  a one-dimensional representation. Since all one-dimensional representations of $G_4$ are obtained from  
$G_4/\langle t \rangle$,  they are trivial on  $\langle t \rangle$.  
In particular, $K$ is invariant under the action of $t$, and therefore a linear combination of the $t$-invariant  monomials of degree three:  
\[
x_0^2x_2, \quad x_0x_1^2, \quad x_1x_3^2 \quad \makebox{and} \quad x_2^2x_3. 
\]
These monomials are cyclically permuted by $s$, which shows that the cubic  $K$ is  (up to a scalar multiple) one of the following: 
\[
K_{\lambda}= x_0^2x_2 + \lambda x_0x_1^2+ \lambda^2 x_1x_3^2+ \lambda^3x_2^2x_3, \quad \makebox{where} \quad 
\lambda \in \lbrace \pm1, ~ \pm i \rbrace. 
\]
The possibilities  $K_{\pm i}$ can be ruled out because 
$p_1 \notin K_{i}$ and $p_2 \notin K_{-i}$. We claim that $I(C_4)_3^{G_4}=0$, which excludes  $K=K_{1}$, and finally implies $K=K_{-1}$. 
To proof the claim, we consider the $G_4$-invariant part of \ref{noetherseq} for $k=3$:
\begin{equation}\label{G4invariant}
0 \to I(C_4)_3^{G_4} \to \mathbb C[x_0, \ldots, x_{3}]_3^{G_4} \to H^0(C,\omega_{C_4}^{\otimes 3})^{G_4} \to 0.
\end{equation}
The space  $\mathbb C[x_0, \ldots, x_{3}]_3^{G_4}$ is 
easily seen to be one-dimensional by computing the inner product 
\[
\langle \Sym^3(\chi_{C_4}), \chi_{triv} \rangle = \frac{1}{|G_4|} \sum_{g \in G_4} \Sym^3(\chi_{C_4})(g)
\]
 with the help of the well known formula:   
\[
\Sym^3(\chi_{C_4})(g)= \tfrac{1}{6}\cdot\left(\chi_{C_4}(g)^3+3\chi_{C_4}(g^2)\chi_{C_4}(g)+ 2\chi_{C_4}(g^3)\right), \qquad g \in G_4. 
\]
\cite[Lemma VI.11, Examples VI.12]{beauAlgSurf} allows us to compute 
\[
h^0(\omega_{C_4}^{\otimes 3})^G= h^0\left(\mathcal O_{\mathbb P^1}\left(-6+2 \left\lfloor \tfrac{9}{4} \right\rfloor + \left\lfloor \tfrac{12}{5} \right\rfloor \right)\right) =h^0(\mathcal O_{\mathbb P^1})=1,
\]
and the exactness of the sequence \ref{G4invariant} yields $I(C_4)_3^{G_4}=0$. 
\end{proof}

\begin{theorem}\label{main2}
    Let $n\geq 2$ and $X_d=(E^{n-1}\times C_d)/G_d$ a rigid quotient. Then, the types of the singularities of $X_d$ and their numbers  are:
        \begin{center}
            {\scriptsize\renewcommand\arraystretch{2.0}\begin{tabular}{|c|c|c|c|c|c|c|c|}
            \hline
                 &   $\tfrac{1}{2}(1,\ldots,1)$ & $\tfrac{1}{3}(1,\ldots,1)$& $\tfrac{1}{3}(1,\ldots,1,2)$ & $\tfrac{1}{4}(1,\ldots,1)$ & $\tfrac{1}{4}(1,\ldots,1,3)$ & $\tfrac{1}{6}(1,\ldots,1)$ & $\tfrac{1}{6}(1,\ldots,1,5)$\\
                \hline
                $d=3$ & -- & $3^{n-1}$ & $3^{n-1}$ & -- & -- & -- & --\\
                \hline
                $d=4$ & $2^{n-1}(2^{n-1}-1)$ &  -- & -- & $2^{n-1}$ & $2^{n-1}$ & -- & --\\
                \hline
                $d=6$ & $\tfrac{2}{3}(4^{n-1}-1)$ & $\tfrac{1}{2}(3^{n-1}-1)$ & $\tfrac{1}{2}(3^{n-1}-1)$ & -- & -- & $1$ & $1$\\
                \hline
            \end{tabular}
            }
        \end{center}

    If $d=6$, then there is precisely one isomorphism class of such quotients.\\
            If $d=3,4$, then there are at most two isomorphism classes, which can be identified under complex conjugation.
\end{theorem}

\begin{proof}
First, we prove the statement about the isomorphism classes.
We start with the group $G_6=\langle 18,3 \rangle \simeq \mathbb Z_3 \rtimes_{\varphi_6} \mathbb Z_6$.
Let $\mathcal S_C$ be the set of generating triples $S$ for $G_6$ of type $[6,6,3]$ such that $S\not\subset G_6\setminus A$.
As usual, $A$ is the unique 
normal abelian subgroup of $G_6$ of index $6$. 
Let $\mathcal S_E$ be the set of generating triples of $G_6$ of type $[2,3,6]$. Let 
\[
\mathfrak X \subset \mathcal S_E^{n-1}\times\mathcal S_C
\]
be the set of $n$-tuples of generating triples which correspond to a rigid action. We show that the action of 
$\Aut(G) \times \mathcal B_3^n \times \mathfrak S_{n-1}$ on this set is transitive, i.e., there is only one orbit and therefore only one 
isomorphism class.
Let $R_C\subset \mathcal S_C$ be a set of representatives for the $(\Aut(G) \times \mathcal B_3)$-action on $\mathcal S_C$ and 
$R_E \subset \mathcal S_E$ be a set of representatives for the $\mathcal B_3$-action on $\mathcal S_E$. 
Clearly, each orbit of the $(\Aut(G) \times \mathcal B_3^n \times \mathfrak S_{n-1})$-action on $\mathfrak X$ has a representative 
in the set  $R_E^{n-1}\times R_C$. Using a MAGMA computation, we see that the set $R_C$ has just one element, which we denote by $S$. Note that the transitivity of the action on $\mathcal S_C$ follows also from the uniqueness of the curve proven in Proposition~\ref{minimalex}. Furthermore, the set $R_E$ has exactly two elements 
$S_1$ and $S_2$; the  characters of the corresponding canonical actions 
 are $\chi_{\zeta_6}$ and $\chi_{\zeta_6^{-1}}$.  According to Corollary \ref{charconds},  the unique tuples in 
$R_E^{n-1}\times R_C$
that  give  rigid actions are  $[S_1,\ldots,S_1,S]$ and $[S_2,\ldots,S_2,S]$. Thus, there are at most two isomorphism classes of rigid quotients.
Again a MAGMA computation shows that 
$[S_1,S]$ and $[S_2,S]$ are in the same orbit under the action of 
$\Aut(G) \times \mathcal B_3^2$. This implies that $[S_1,\ldots,S_1,S]$ and $[S_2,\ldots,S_2,S]$ are equivalent under the action of $\Aut(G) \times \mathcal B_3^n \times \mathfrak S_{n-1}$, and therefore, there is a unique isomorphism class of rigid quotients,  see Corollary \ref{bihollift}. 

Now, let $d=3,4$. Also here,  following the above strategy, only two tuples of generating triples remain, $[S_1,\ldots,S_1,S]$  and $[S_2,\ldots,S_2,S]$. However, in this case, they belong to different orbits under the action of $\Aut(G) \times \mathcal B_3^n \times \mathfrak S_{n-1}$. But the complex conjuagte of $[S_1,\ldots,S_1,S]$ lies in the same orbit as the second tuple. This means that the first quotient is isomorphic to the complex conjugate of the second quotient. 

The reader can find the source code for the MAGMA computations on the webpage 
\begin{center}
\url{http://www.staff.uni-bayreuth.de/~bt300503/publi.html}.
\end{center}

The computation of the singularities can be done similarly as in the proof of Theorem~\ref{quot366} (see also \cite{BG20} for the case $d=3$).
\end{proof}

\begin{openprob}
It is not clear to us if the complex conjugate  varieties in Theorem \ref{main2} (for $d=3,4$) are biholomorphic or not. 
Corollary \ref{bihollift} just implies that there is no biholomorphism lifting to a 
biholomorphism of the form 
\begin{align*}
E^{n-1}\times C_d &\longrightarrow   E^{n-1}\times C_d\\
(z_1,\ldots,z_n)&\longmapsto (u_1(z_{\tau(1)}),\ldots,u_{n-1}(z_{\tau(n-1)}), u_{n}(z_{n})),
\end{align*}
for some permutation $\tau\in\mathfrak S_{n-1}$.
However there might be other biholomorphisms that are not of product type or do not even lift.  
\end{openprob}


\section{Proof of Theorem~\ref{theo:Resolutions}}\label{sec:resolutions}

This section is devoted to the proof of Theorem~\ref{theo:Resolutions}. For this, we construct resolutions of the rigid singular quotients $X$ of the form 
\[
(E^{n-1}\times C)/G, \qquad \makebox{where} \qquad G= A\rtimes_{\varphi_d}\ZZ_d,
\]
preserving the rigidity in order to obtain rigid projective \emph{manifolds}.\\
If $n\geq d$, then all singularities of the quotients are canonical (cf. Corollary~\ref{cor:sing}), which implies that the Kodaira dimension of the resolutions is the same as the Kodaira dimension of the product $E^{n-1}\times C$, namely $1$.

Leray's spectral sequence can be used to derive sufficient conditions on a resolution $\rho\colon\hat{X}\to X$ to guarantee that $\hat{X}$ is still infinitesimally rigid:

\begin{proposition}[{\cite{BG20}, Proposition~2.10}]\label{prop:CondRes}
	Let $Y$ be a projective manifold and $G\leq\Aut(Y)$ be a finite group acting freely in codimension one. Let $\rho\colon \hat{X}\to X$ be a resolution of the quotient $X=Y/G$ such that
	\begin{enumerate}
		\item $\rho_\ast\Theta_{\hat{X}}\simeq \Theta_X$,
		\item $R^1\rho_\ast\Theta_{\hat{X}}=0$.
	\end{enumerate}
	Then, $H^1(\hat{X},\Theta_{\hat{X}})\simeq H^1(X,\Theta_X)\simeq H^1(Y,\Theta_Y)^G$.
\end{proposition}

By Corollary~\ref{cor:sing}, our quotients have only isolated singularities. Thus, the construction of resolutions with those properties is a local problem. Furthermore, the singularities are all cyclic quotient singularities of type
\[
\frac{1}{\ell}(1,\ldots,1)\qquad\makebox{or}\qquad \frac{1}{\ell}(1,\ldots,1,\ell-1), \qquad\makebox{where}\qquad\ell\geq 2.
\]
Those singularities are known to be represented by affine toric varieties with cone $$\sigma:=\cone(e_1,\ldots, e_n)\subset\RR^n$$ and lattice
\[N_j:=\ZZ^n +\ZZ\cdot \tfrac{1}{\ell}(1,\ldots,1,a_j)\subset \RR^n,\]
where $a_1=1$ and $a_2=\ell-1$.\\
This allows us to use tools from toric geometry to construct such a resolution. The basic references for toric geometry are the textbooks \cite{F93} and \cite{CLS11}.\\
For the singularities of type $\tfrac{1}{\ell}(1,\ldots,1),\:\ell\geq 2$, we can use the toric blowup as a resolution. The verification of the conditions outlined in Proposition~\ref{prop:CondRes} is analogous to \cite{BG20}.\\
For the singularities of the other type, we generalize the construction and the proof of \cite{BG20}, where only the case $\ell=3$ was settled, to prove the following:

\begin{proposition}\label{prop:Resolution}
	Let $n\geq 3$, $\ell\geq 2$, and let $U$ be the affine cyclic quotient singularity of type $\tfrac{1}{\ell}(1,\ldots,1,\ell-1)$. Then, there exists a resolution $\rho\colon \hat{U}\to U$ of singularities with the following properties:
	\begin{enumerate}
		\item $\rho_\ast\Theta_{\hat{U}}\simeq\Theta_{U}$\: and
		\item $R^1\rho_\ast\Theta_{\hat{U}}=0.$
	\end{enumerate}
\end{proposition}

For the construction of the resolution, we use the toric description from above (set $N:=N_2$) and consider the finite sequence of star subdivisons of the cone $\sigma$ along the rays generated by the elements
\[v_k:=\tfrac{1}{\ell}(k,\ldots,k,\ell-k),\quad k=1,\ldots,\ell-1.\]
This yields a fan $\Sigma$ with maximal cones
\begin{align*}
	\begin{split}
		\sigma_i^{(0)}&:=\cone(e_1,\dots,\widehat{e_i},\dots,e_n,v_1),\hspace{1,7cm} i=1,\dots,n-1,\\
		\sigma_i^{(k)}&:=\cone(e_1,\dots,\widehat{e_i},\dots,e_{n-1},v_k,v_{k+1}),\quad i=1,\dots,n-1;\:k=1,\dots,\ell-2,\\
		\sigma_n&:=\cone(e_1,\dots,e_{n-1},v_{\ell-1}).
	\end{split}	
\end{align*}
\begin{figure}[h] 
  \centering
     \includegraphics[width=0.5\textwidth]{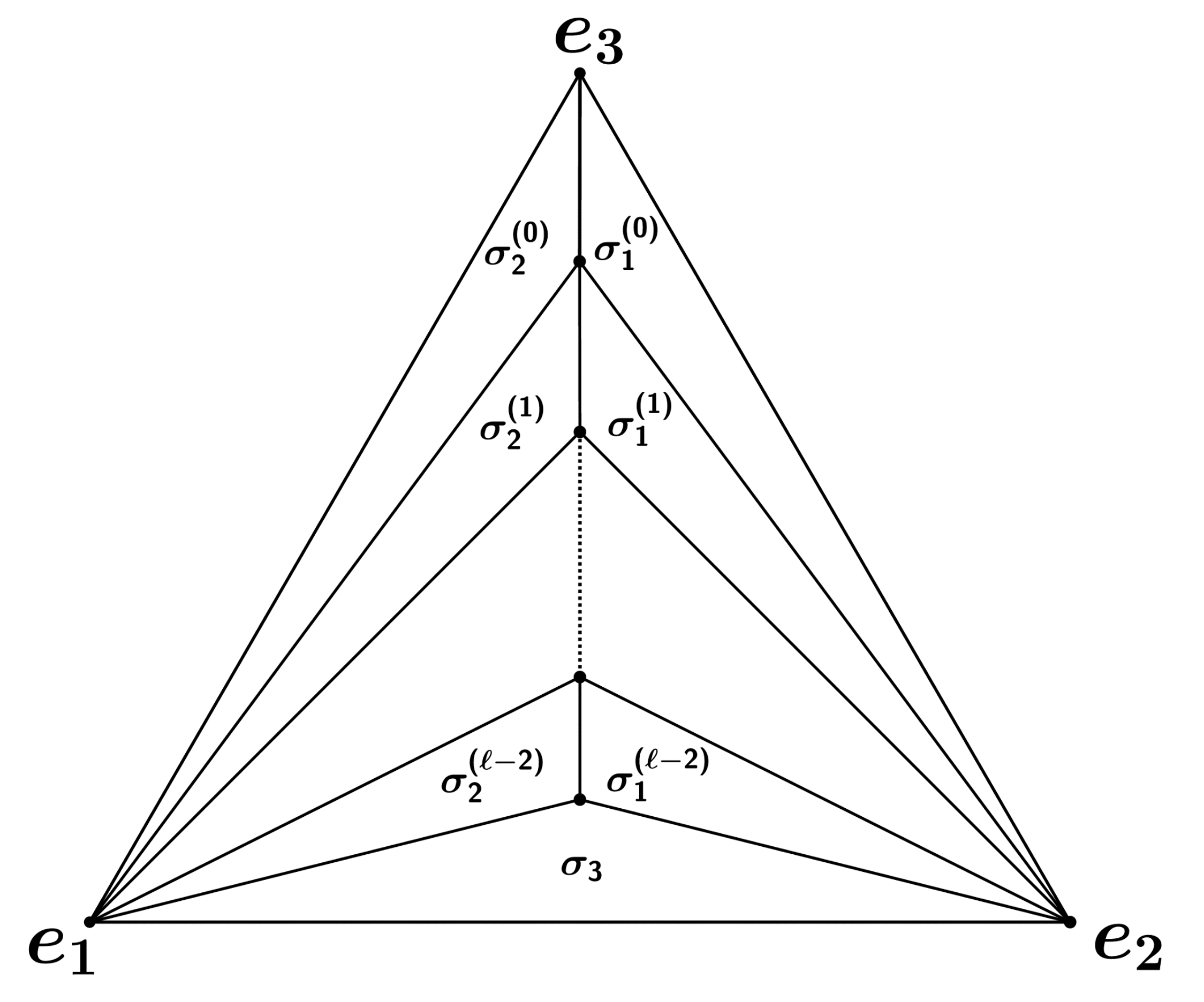}
\end{figure}

It is an easy computation to show that all these maximal cones are smooth. Therefore, the subdivisons induce a resoultion $\rho\colon U_\Sigma\to U$, where $U_\Sigma$ is the toric variety of the fan $\Sigma$. The resolution admits $\ell-1$ exceptional divisors $E_1,\ldots,E_{\ell-1}$ corresponding to the additional rays generated by $v_k,\:k=1,\ldots,\ell-1$. We denote the divisors corresponding to the rays $\RR_{\geq 0}e_i$ by $D_i$.

\begin{proposition}\label{prop:excdivisor}
	The exceptional divisors of the toric resolution can be described as follows:
	\begin{enumerate}
		\item The exceptional prime divisor $E_{\ell-1}$ is isomorphic to $\PP^{n-1}$.
		\item For $k=1,\ldots,\ell-2,$ the exceptional prime divisor $E_k$ is isomorphic to the projective bundle
		\[pr\colon E_k\simeq \PP(\Oh_{\PP^{n-2}}\oplus\Oh_{\PP^{n-2}}(\ell-k))\to\PP^{n-2}.\]
		In particular,
		\[\omega_{E_k}\simeq pr^*\Oh_{\PP^{n-2}}(-n+2)\otimes\Oh_{E_k}(E_k).\]
	\end{enumerate}
\end{proposition}

\begin{proof}
	We leave the proof for the divisor $E_{\ell-1}$ to the reader as the reasoning behind it is similar to the following proof for the other divisors (but easier). Let $k$ be in $\{1,\ldots,\ell-2\}$. As a toric variety, $E_k$ is given by the quotient lattice $N(v_k)=N/\ZZ v_k$ and quotient cones
	\[\overline{\sigma_i^{(k-1)}},\:\overline{\sigma_i^{(k)}}\subset N(v_k)_ \RR,\quad 1\leq i\leq n-1\]
	together with their faces. Denote with $u_1,\ldots,u_{n-1}$ the standard basis of $\ZZ^{n-1}$ and define $e:=u_{n-1}$ and $u_0:=-(u_1+\ldots+u_{n-2})$. The computation for $k=1$ is the same as in \cite[Proposition~5.5]{BG20}. In the case $k>1$, the quotient lattice is generated by the classes $[e_2],\ldots,[e_{n-1}],[v_{k-1}]$, and we can identify the lattice with $\ZZ^{n-2}\times \ZZ$ via the isomorphism
	\[\phi\colon N(v_k)\longrightarrow\ZZ^{n-2}\times\ZZ,\quad [e_i]\mapsto u_{i-1},\,[v_{k-1}]\mapsto -e.\]
	One can easily show that we can find integers $\mu$ and $\lambda$ such that \[e_1=-e_2-\ldots-e_{n-1}+(k-\ell)\cdot v_{k-1}+\mu\cdot v_k\quad\mathrm{and}\quad v_{k+1}=-v_{k-1}+\lambda\cdot v_k.\]
	Therefore, we get $\phi([e_1])=u_0+(\ell-k)\cdot e$ and $\phi([v_{k+1}])=e$. Using the $\RR$-linear extension of $\phi$, which identifies $N(v_k)\otimes\RR$ with $\RR^{n-1}$, the quotient cones, viewed as cones in $\RR^{n-1}$, are given by:
	\begin{align*}
	\overline{\sigma_i^{(k-1)}}&\simeq \cone(u_0+(\ell-k)\cdot e,u_1,\dots,\widehat{u_{i-1}},\dots,u_{n-2},-e),\\
	\overline{\sigma_i^{(k)}}&\simeq\cone(u_0+(\ell-k)\cdot e,u_1,\dots,\widehat{u_{i-1}},\dots,u_{n-2},e).
	\end{align*}
	According to \cite[Example 7.3.5]{CLS11}, these cones and their faces build the fan of the projective bundle $E_k\simeq\PP(\Oh\oplus\Oh(\ell-k))$.\\
	Finally, we compute the canonical bundle of $E_k$ for $k=1,\ldots,\ell-2$ as follows: By using the adjunction formula and \cite[Theorem~8.2.3]{CLS11}, we get:
	\[\omega_{E_k}\simeq\Oh_{E_k}(-D_1-\ldots-D_n-E_1-\ldots-\widehat{E_k}-\ldots-E_{\ell-1}).\]
	Since $pr^\ast\Oh_{\PP^{n-2}}(1)\simeq\Oh_{E_k}(D_2)$, compare \cite[Proposition~6.2.7]{CLS11}, and
	\[0\sim_{lin}\divi(e_1-(n-3)\cdot e_2+e_3+\ldots+e_n)=D_1-(n-3)\cdot D_2+D_3+\ldots+D_n+E_1+\ldots+E_{\ell-1},\]
	we conclude
	\[\omega_{E_k}\simeq \Oh_{E_k}(-(n-2)\cdot D_2+ E_k)\simeq pr^*\Oh_{\PP^{n-2}}(-n+2)\otimes\Oh_{E_k}(E_k).\qedhere \]
\end{proof}

\begin{proof}[Proof of Proposition~\ref{prop:Resolution}]
	Primarily, we prove that the resolution $\rho\colon U_\Sigma\to U$ satisfies the condition $\rho_\ast\Theta_{U_\Sigma}\simeq\Theta_{U}$.
	Let $D_i\subset U_\Sigma$ and $D_i'\subset U$ be the divisors corresponding to the rays generated by $e_i$. As proved in \cite[Proposition~5.8]{BG20}, we have to show that
	\[P_{D_i}\cap N^\vee=P_{D_i'}\cap N^\vee\]
	is true for all $i=1,\ldots,n$, where $P_{D_i}$ and $P_{D_i'}$ are the polyhedra of the divisors $D_i$ and $D_i'$. These polyhedra are given by the following conditions:
	\begin{align*}
		P_{D'_i}&=\{x\in\RR^n\mid x_i\geq -1,\, x_j\geq 0,\: j\neq i\}\quad\quad\mathrm{and}\\
		P_{D_i}&=P_{D'_i}\cap\{\langle x,v_k\rangle\geq 0, \:k=1,\dots,\ell-1\}=\\
		&=P_{D'_i}\cap\{kx_1+\ldots+kx_{n-1}+(\ell-k)x_n\geq 0,\:k=1,\dots,\ell-1\}.
	\end{align*}
	The dual latice $N^\vee$ consists of all points $x\in\ZZ^n$ satisfying that the sum $x_1+\ldots+x_{n-1}+(\ell-1)x_n$ is divisible by $\ell$. Obviously, the integral points of $P_{D_i}$ are contained in the polyhedron $P_{D_i'}$. So let $x$ be an element of $P_{D_i'}\cap N^\vee$ and $k\in\{1,\ldots,\ell-1\}$. Since
	\[k(x_1+\ldots+x_{n-1}+(\ell-1)x_n)=kx_1+\ldots+kx_{n-1}+(\ell-k)x_n+\ell(k-1)x_n\]
	and since the left hand side of these equation is divisible by $\ell$, the sum $kx_1+\ldots+kx_{n-1}+(\ell-k)x_n$ is divisible by $\ell$ as well. By assumption, we know that this sum is an integer at least $-(\ell-1)$. Therefore, it must be greater or equal to zero, and we conclude that $x$ is an element of $P_{D_i}$, too.\\
	
	It remains to show that $R^1\rho_\ast\Theta_{U_\Sigma}=0$. The arguments are similar to the proof of Proposition~5.10 in \cite{BG20}. Using the toric Euler sequence and the fact that $U_\sigma$ has rational singularities, one can show that this is the case if and only if
	\begin{enumerate}
		\item $R^1\rho_*\Oh_{U_\Sigma}(D_i)=0$ for all $i=1,\dots,n$,
		\item $R^1\rho_*\Oh_{E_k}(E_k)=0$ for all $k=1,\dots,\ell-2$ and
		\item $R^1\rho_*\Oh_{E_{\ell-1}}(E_{\ell-1})=0$.
	\end{enumerate}
	The vanishing of $R^1\rho_*\Oh_{E_{\ell-1}}(E_{\ell-1})$ is clear since $E'\simeq\PP^{n-1}$, which implies that the bundle $\Oh_{E_{\ell-1}}(E_{\ell-1})$ is a multiple of $\Oh_{\PP^{n-1}}(1)$, whose first cohomology vanishes ($n\geq 3$).\\
	By symmetry, it is enough to consider the cases $i=1,n$ to show (1). We start with $D_1$. The Cartier data of $D_1$ is the collection of the vectors
	\begin{align*}
		m_{\sigma_n}=-e_1+(\ell-1)\cdot e_n,\quad 
		m_{\sigma_1^{(k)}}=0\quad\mathrm{and}\quad m_{\sigma_i^{(k)}}=-e_1+e_i,
	\end{align*}
	for $i=2,\ldots,n-1$ and $k=0,\ldots,\ell-1$. Since all these vectors are elements of the polyhedron $P_{D_1}$, the sheaf $\Oh_{U_\Sigma}(D_1)$ is globally generated according to \cite[Proposition~6.1.1]{CLS11}. Using Demazure vanishing \cite[Theorem~9.2.3]{CLS11}, we conclude $R^1\rho_*\Oh_{U_\Sigma}(D_1)=0$.\\
	The Cartier data of the divisor $D_n$ is the following collection:
	\begin{align*}
		m_{\sigma_n}=0,\quad
		m_{\sigma_i^{(0)}}=-e_n+(\ell-1)\cdot e_i\quad\mathrm{and}\quad
		m_{\sigma_i^{(k)}}=0
	\end{align*}
	for $i=1,\ldots,n-1$ and $k=1,\ldots, \ell-1$. These elements are all contained in the polyhedron associated to $D_n$, and we can argue as above.\\
	Finally, let $k\in\{1,\ldots,\ell-2\}$. By Proposition~\ref{prop:excdivisor}, the canonical bundle of $E_k$ is
	\[\omega_{E_k}\simeq pr^*\Oh_{\PP^{n-2}}(-n+2)\otimes\Oh_{E_k}(E_k).\]
	Using Serre duality on $E_k$ and the projection formula, we can conclude:
		\begin{align*}
		H^1(E_k,\Oh_{E_k}(E_k))	&\simeq H^{n-2}(E_k,pr^*\Oh_{\PP^{n-2}}(-n+2))^\vee
		\simeq H^{n-2}(\PP^{n-2},\Oh_{\PP^{n-2}}(-n+2))^\vee\\
		&\simeq H^0(\PP^{n-2},\Oh_{\PP^{n-2}}(-1))=0. \qedhere
	\end{align*}
\end{proof}

{\tiny MATHEMATISCHES INSTITUT, UNIVERSIT\"AT BAYREUTH, 95447 BAYREUTH, GERMANY}

{\scriptsize\emph{E-mail address: Ingrid.Bauer@uni-bayreuth.de,}} \quad {\scriptsize\emph{Christian.Gleissner@uni-bayreuth.de,}} \quad
{\scriptsize\emph{Julia.Kotonski@uni-bayreuth.de}}

\end{document}